\documentclass{amsart}
\usepackage{amsmath}
\usepackage{amsthm}
\usepackage{amsfonts,mathrsfs}
\usepackage{amssymb}
\usepackage{graphicx}
\usepackage{enumitem}
\usepackage{color}
\usepackage{verbatim}
\usepackage{enumerate}
\usepackage{amssymb,latexsym}
\usepackage{cancel}
\usepackage{cite}
\usepackage{lineno}
\usepackage{hyperref}
\usepackage{amscd}
\usepackage{lineno}

\theoremstyle{plain}
\newtheorem{theorem}{Theorem}[section]
\newtheorem{lemma}[theorem]{Lemma}
\newtheorem{remark}[theorem]{Remark}

\theoremstyle{definition}
\newtheorem{definition}[theorem]{Definition}

\newtheoremstyle{TheoremNum}
{\topsep}{\topsep}              
{\itshape}                      
{}                              
{\bfseries}                     
{.}                             
{ }                             
{\thmname{#1}\thmnote{ \bfseries #3}}

\newcommand{\F}{\mathbb F}
\newcommand{\fqn}{\mathbb{F}_{q^n}}
\newcommand{\fqk}{\mathbb{F}_{q^k}}
\newcommand{\fq}{\mathbb{F}_{q}}

\newcommand{\rank}{\mathrm{rank}}

\newcommand{\cS}{\mathcal S}
\newcommand{\cC}{\mathcal C}

\newcommand{\sS}{\mathscr S}

\newcommand{\Aut}{\mathrm{Aut}}

\newcommand{\GL}{\mathrm{GL}}
\newcommand{\Gal}{\mathrm{Gal}}
\newcommand{\Tr}{ \ensuremath{ \mathrm{Tr}}}
\newcommand{\cT}{\mathcal{T}}

\newcommand{\RN}[1]{%
	\textup{\uppercase\expandafter{\romannumeral#1}}%
}

\def\zhou#1 {\fbox {\footnote {\ }}\ \footnotetext { From Yue: {\color{blue}#1}}}

\def\chen#1 {\fbox {\footnote {\ }}\ \footnotetext { From Tang: {\color{red}#1}}}

\begin{document}
	\title[New maximum linear symmetric rank-distance codes]{A new family of maximum linear symmetric rank-distance codes}
	\author[W. Tang]{Wei Tang\textsuperscript{\,1}}
	\author[Y. Zhou]{Yue Zhou\textsuperscript{\,2}}
	\address{\textsuperscript{1}Purple Mountain Laborotories, 211111 Nanjing, China}
	\address{\textsuperscript{2}Department of Mathematics, National University of Defense Technology, 410073 Changsha, China}
	\email{yue.zhou.ovgu@gmail.com}
	\begin{abstract}
		Let $\mathscr{S}_n(q)$ denote the set of symmetric bilinear forms over an $n$-dimensional $\fq$-vector space. A subset $\mathcal{C}$ of $\mathscr{S}_n(q)$ is called a $d$-code if the rank of $A-B$ is larger than or equal to $d$ for any distinct $A$ and $B$ in $\mathcal{C}$. If $\mathcal{C}$ is further closed under matrix addition, then $|\mathcal{C}|$ is sharply upper bounded by $q^{n(n-d+2)/2}$ if $n-d$ is even and $q^{(n+1)(n-d+1)/2}$ if $n-d$ is odd. Additive codes meeting these upper bounds are called maximum. There are very few known constructions of them. In this paper, we obtain a new family of maximum $\fq$-linear $(n-2)$-codes in $\mathscr{S}_n(q)$ for $n=6,8$ and $10$ which are not equivalent to any known constructions. Furthermore, we completely determine the equivalence between distinct members in this new family.
	\end{abstract}
	\maketitle
	
	\section{Introduction}\label{sec:intro}
	Let $q$ be a prime power.  For any two matrices $A, B$ of the same size over $\F_q$, their \emph{rank-distance} is defined as 
	\[d_r(A,B)=\rank(A-B).\]
	
	Let $V=V(n,q)$ be an $n$-dimensional $\fq$-vector space. Let $\sS_n(q)$ denote the set of symmetric bilinear forms on $V$.
	For any given integer $d$ satisfying $1\leq d\leq n$, a subset $\cC \subseteq \sS_n(q)$ is called a \emph{$d$-code} if
	$d_r(A, B)\geq d$ for all different $A, B \in \cC$, where we implicitly assume that set $\cC$ contains at least two elements. If a $d$–code $\cC \subseteq \sS_n(q)$ forms a subgroup of $\sS_n(q)$ with respect to the addition of matrices, then $\cC$ is called \emph{additive}.

	The following theorem was established in \cite{Schmidt_2015} for odd $q$, in \cite{Schmidt_2010} for the case where $q$ is even and $d$ is odd, and in \cite{Schmidt_2020} when both $q$ and $d$ are even.
	\begin{theorem}\label{the1.2}
		Let $\cC$ be a $d$-code in $\sS_n(q)$. When $d$ is even, $\cC$ is further required to be additive. Then we have following tight upper bound on the size of $\cC$,
		\begin{equation}\label{eq:bound}
			\mid\cC\mid\leq	
			\begin{cases}
				q^{n(n-d+2)/2}, & \text{ if }n-d\text{ is even};\\
				q^{(n+1)(n-d+1)/2}, & \text{ if }n-d\text{ is odd}.
			\end{cases}
		\end{equation}
	\end{theorem}
	An additive $d$-code $\cC\subseteq \sS_n(q)$ whose size meets the upper bound in \eqref{eq:bound} is called \emph{maximum}. In fact, it is originally called \emph{maximal} in \cite{Schmidt_2010,Schmidt_2015,Schmidt_2020}. We adopt the term ``maximum" in order to remain consistent with \emph{maximum rank-metric code} (MRD code, for short) which defined to be a subset $|\cC|\subseteq \F_{q}^{m\times n}$ of minimum rank distance $d$ satisfying $\cC=q^{\min\{m,n\}(\max\{m,n\}-d+1)}$; for a survey on MRD codes, we refer to \cite{sheekey_MRD_survey}.
	
	No matter $n-d$ is even or odd, the upper bound in \eqref{eq:bound} can be achieved by concrete constructions which will be introduced later in this section.
	
	When $d$ is even, the upper bound on the sizes of non-additive $d$-codes becomes mysterious. General upper bounds for them can be found in \cite[Proposition 3.7]{Schmidt_2015} and \cite[Proposition 3.4]{RMC}. There do exist non-additive $d$-codes that contains more elements than maximum additive d-codes. For $n=3$ and $q > 2$, an infinite family of $2$–code of size $q^4 + q^3 + 1$ is provided in \cite{cossidente_symmetric_2022}. More sporadic examples can be bound in \cite[Table 2 and Table 9]{RMC}.
	
	We focus on maximum additive $d$-codes in $\sS_n(q)$ in this paper.
	So far,  there are only two known families of them in $\sS_n(q)$ for arbitrary $d$ up to equivalence. Both were discovered by Schmidt in \cite{Schmidt_2010,Schmidt_2015}.
	
	Denote \( \Tr(\cdot) \) as the trace function from \( \mathbb{F}_{q^n} \) to \( \mathbb{F}_q \).
	The first family is defined as follows. For any integer $1\leq d\leq n$ such that $n-d$ is even and for any $s$ coprime with $n$, consider the following subset of $\sS_n(q)$:
	\begin{equation}\label{eq:S_1}
		\cS_{n,d,s}=\left\{\Tr\left(b_0xy+\sum_{i=1}^{(n-d)/2}b_i\left(x^{q^{si}}y+y^{q^{si}}x\right)\right): b_0,\dots,b_{(n-d)/2}\in\fqn \right\}.	
	\end{equation}
	It is easy to see that there are exactly $q^{n(n-d+2)/2}$ elements in $\cS_{n,d,s}$ which meets the upper bound in \eqref{eq:bound}.
	
	The second family is in fact derived from the one above by the following secondary construction. Let $W$ be an $(n-1)$-dimensional subspace of $V$. For any subset $Y$,  define the punctured set with respect to $W$ of $Y$ to be
	$$Y|_W := \{B|_W : B\in Y\}$$
	where $B|_W$ is the restriction of $B$ onto $W$. As the rank of $B|_W$ is at most two less than the rank of $B$, $Y|_W$ is a maximum $d$-code of $\sS_n(q)$ with $n-d$ an odd integer provided that $Y$ is a $(d+2)$-code of $\sS_{n+1}(q)$ \cite[Theorem 4.1]{Schmidt_2015}.
	
	For special value of $d$, there are more constructions. When $d=n$, a maximum additive $n$-code in $\sS_n(q)$ corresponds to a symplectic semifield of order $q^n$, which is also equivalent to a commutative semifield under the Knuth operation. The study of (commutative) semifields was initiated by Dickson in \cite{dickson_commutative_1906} around 120 years ago. There are many constructions of semifields; see \cite{lavrauw_semifields_2011}. Among them, there are a few families of commutative semifields: for $q$ even, there is only one family which contains many inequivalent members, and they are constructed by Kantor in \cite{kantor_commutative_2003} which generalizes the one found by Knuth in \cite{knuth_class_1965}; for $q$ odd, besides those families mentioned in \cite{lavrauw_semifields_2011}, more new constructions are summarized in \cite{gologlu_exponential_2023}.
	
	For $d=2$ and $n>2$, two more constructions which are not equivalent to any previously known ones can be found in \cite{LLTZ} and \cite{zhou}.
	
	The construction of symmetric rank-metric codes using $q$-polynomials was pioneered by Schmidt, who established the fundamental framework \cite{Schmidt_2010}. In \cite{LLTZ}, the authors introduced the idea of adding \emph{twisted} terms to $q$-polynomials for new constructions of maximum symmetric rank-metric codes. 
		In this paper, we will follow this basic idea and obtain a new family of maximum symmetric $(n-2)$-codes.
	\begin{theorem}\label{th:main}
		For any positive integer $k$, let $n=2k$ and $s$ be an integer satisfying $0<s<2k$, $\gcd(s,2k)=1$. Let $q$ be an odd prime power. For any non-square $\eta\in \fqn$, define the following set of symmetric bilinear forms
		\begin{align*}
			\cT_{n,s,\eta}&=\left\{\Tr\left(b_0x^{q^{k}}y+b_1\left(x^{q^{s(k-1)}}y+y^{q^{s(k-1)}}x\right)+\eta b_2\left(x^{q^{ s(k-2)}}y+y^{q^{s(k-2)}}x\right)\right) \right.:\\ 
			&\left.		b_0, b_2\in\F_{q^k},b_1\in\F_{q^{2k}}\right\}.
		\end{align*}
		Then $\cT_{n,s,\eta}$ is a maximum $\fq$-linear $(n-2)$-code for $k=3,4$ and $5$.
	\end{theorem}
	
	The rest parts of this paper are organized as follows: In Section \ref{sec:pre} we revise some basic concepts and results about linearized polynomials and the equivalence defined over symmetric rank-distance codes. Then we prove the proof of Theorem \ref{th:main} in Section \ref{sec:proof_main}. The proof that the codes constructed in Theorem \ref{th:main} are not equivalent to any known ones given in Section \ref{sec:equi} in which we also completely determine the equivalence between different members in the family obtained in Theorem \ref{th:main}. Section \ref{sec:conclu} concludes the paper with an open question.
	
	\section{Preliminaries}\label{sec:pre}
	
	For any prime power $q$ and positive integer $n$, a polynomial
	\[
	f:= \sum_{i=0}^m a_i X^{q^i}
	\]
	with $a_i\in \fqn$,
	is called a \emph{$q$-polynomial} or a \emph{linearized polynomial}. If $a_m\neq 0$, then we say the \emph{$q$-degree} of $f$ is $m$.
	It is easy to see that the map $x\mapsto f(x)$ is $\F_q$-linear over $\F_{q^n}$ for every $q$-polynomial $f$. Hence one can define the \emph{rank} of $f$ which equals the rank of the corresponding $\F_q$-linear map.
	
	It is well known that $\left\{ \sum_{i=0}^{n-1} a_i X^{q^i}: a_i\in \fqn \right\}$ under addition and multiplication modulo $X^{q^n}-X$ is isomorphic to the matrix ring $\F_q^{n\times n}$.
	
	Given a $q$-polynomial $f$ of $q$-degree $m$, it is clear that the $\F_q$-dimension of the kernel of the $\F_q$-linear map $x\mapsto f(x)$ is upper bounded by $m$. In fact, one can also generalize this result to the following shape.
	
	\begin{theorem}\label{th:Gow}\cite[Theorem 5]{gow_galois_2009}	
		Let $\mathbb{L}$ be a cyclic extension of a field $\F$ of degree $n$, and suppose that $\sigma$ generates the Galois group of $\mathbb{L}$ over $\F$. Let $k$
		be an integer satisfying $1 \leq k \leq n$, and let $a_0, a_1,\cdots, a_k$ be elements of $\mathbb{L}$,
		not all of them are zero. Then the $\F$-linear transformation of $\mathbb{L}$ defined as
		$$f(x)=a_0x+a_1x^{\sigma}+\cdots+a_kx^{\sigma^k},$$
		has kernel with dimension at most $k$ in $\mathbb{L}$.
	\end{theorem}
	\begin{lemma}\label{le:k,a_i}
		Let $n=2k$ and $a_i\in\fqn,i=0,1,\cdots, k-1$. Suppose that $\sigma$ generates $\Gal(\fqn/\fq)$. If  $\sum_{i=0}^{k-1}a_ix^{\sigma^i}=0$ for all $x\in\fqk$, then $a_i=0,i=0,1,\cdots, k-1$.		
	\end{lemma}
	\begin{proof}
		As the $\F_q$-linear map $x\mapsto\sum_{i=0}^{k-1}a_ix^{\sigma^i}$ has kernel with dimension at least $k$ in $\fqn$, then by Theorem \ref{th:Gow}, we get $a_i=0,i=0,1,\cdots, k-1$.
	\end{proof}
	For $f=\sum_{i=0}^{n-1}a_i X^{q^i}$ with $a_0,a_1,\ldots,a_{n-1}\in\F_{q^n}$ and let $D_f$ denote the associated \emph{Dickson matrix} (or \emph{$q$-circulant matrix})
	\[D_f:=
	\begin{pmatrix}
		a_0 & a_1 & \ldots & a_{n-1} \\
		a_{n-1}^q & a_0^q & \ldots & a_{n-2}^q \\
		\vdots & \vdots & \vdots & \vdots \\
		a_1^{q^{n-1}} & a_2^{q^{n-1}} & \ldots & a_0^{q^{n-1}}
	\end{pmatrix}.
	\]
	
	\begin{lemma}\cite[Proposition 4.4]{wu_linearized_2013}\label{lem:dickson_rank}
		For any $q$-polynomial $f$ of $q$-degree no larger than $n-1$, the rank of $f$ equals $\rank(D_f)$.
	\end{lemma}
	
	For a $q$-polynomial $f = \sum_{i=0}^{n-1} a_i X^{q^i}$, its \emph{adjoint} polynomial $\hat{f}$  is defined as
	\[
	\hat{f} := a_0X + \sum_{i=1}^{n-1} a_i^{q^{n-i}} X^{q^{n-i}}. 
	\]
	One of the most important property of $\hat{f}$ is the following
	\[
	\Tr(f(x)y) = \Tr(x\hat{f}(y))
	\]
	for any $x,y\in \fqn$, where $\Tr(\cdot)$ stands for the absolute trace function from $\fqn $ to $\F_q$. It is obvious that $\Tr(f(x)y)$ is a symmetric bilinear form if and only if $f=\hat{f}$.
	
	For given $a \in \fq^*$, $\sigma\in \Aut(q)$, invertible matrix $P\in \F_q^{n\times n}$ and $\sS_0 \in \sS_n(q)$, we define $\Phi:\sS_n(q)\rightarrow \sS_n(q)$ by
	\begin{equation}\label{eq:phi}
		\Phi (C)=aP^TC^\sigma P+\sS_0.
	\end{equation}
	where $P^T$ stands for the transpose of $P$ and $C^\sigma=(c_{ij}^\sigma)$ for $C= (c_{ij})$. It is readily verified that the map $\Phi$ preserves the rank-distance on $\sS_n(q)$. In fact, the converse statement is also true except for the case that $q=2$ and $n=3$; see \cite{WZ}. 
	\begin{definition}
		For two subsets $\cC_1$ and $\cC_2$ of $\sS_n(q)$, if there exists $\Phi$ defined by \eqref{eq:phi} for certain $a\in \F_q^*$, $P\in \GL(n,q)$, $\sigma\in \Aut(\F_q)$ and $\sS_0\in \sS_n(q)$  such that
		\[\cC_2=\{\Phi(C):C\in \cC_1\},\]
		then we say $\cC_1$ and $\cC_2$ are \emph{equivalent} in $\sS_n (q)$ and we write $\cC_1 \cong \cC_2$.
	\end{definition}
	
	It is not difficult to see that if $\cC_1$ and $\cC_2$ are additive, we may assume $\sS_0=0$ in the definition above. If they are both $\F_q$-linear, we may further assume that $a=1$.

		Given an additive subgroup (additive code) $\mathcal{C} \subseteq \mathscr{S}_n(q)$, its \emph{Delsarte dual} $\mathcal{C}^{\perp}$ is defined as
		\[
		\mathcal{C}^{\perp} := \left\{ D \in \mathscr{S}_n(q) : \operatorname{tr}(C D) = 0 \ \text{for all } C \in \mathcal{C} \right\},
		\]
		where $\mathrm{tr}(M)$ stands for the trace of matrix $M$; see \cite{Schmidt_2015} for more details.
		
		There exists a natural representation mapping symmetric matrices to linearized polynomials. Under this representation, the set of symmetric matrices $\mathscr{S}_n(q)$ corresponds precisely to the set of all self-adjoint linearized polynomials of $q$-degree up to $n-1$. We denote this polynomial set by $W$:
		\[
		W= \left\{f = \sum_{i=0}^{n-1} f_i X^{q^i} \in \mathbb{F}_{q^n}[X] : \hat{f}=f \right\}.
		\]
		Under this correspondence, matrix transposition corresponds precisely to taking the adjoint of a linearized polynomial.  
		
		We define an $\mathbb{F}_q$-bilinear form (inner product) on $W$ by
		\[
		\langle f, g \rangle := \Tr\left(\sum_{i=0}^{n-1} f_i g_i \right),
		\]
		for $f=\sum_{i} f_i X^{q^{i}}, g=\sum_{i} g_i X^{q^{i}} \in W$.
		The $\F_q$-linear space $W$ together with the inner product defined as above is essentially equivalent to $\mathscr{S}_n(q)$ equipped with the inner product $\mathrm{tr}(AB)$ for $A,B\in \mathscr{S}_n(q)$.
		
		Hence, for an additive subset $\mathcal{C} \subseteq W$, its \emph{Delsarte dual} can be  defined as
		\begin{equation}\label{eq:dual}
			\mathcal{C}^{\perp} := \left\{ f \in W \mid \langle f, g \rangle = 0 \ \text{for all } g \in \mathcal{C} \right\}.
		\end{equation}
	
	\section{Proof of Theorem \ref{th:main}}\label{sec:proof_main}
	
	This section presents the complete proof of Theorem \ref{th:main}.
	
	It is clear that there are exactly $q^{2n}$ elements in $\cT_{n,s,\eta}$ which meets the upper bound given in \eqref{eq:bound}. Thus we only have to show that the rank of any nonzero symmetric bilinear form in $\cT_{n,s,\eta}$ is larger than or equal to $n-2$.
	
	For any element $T(x,y) \in \cT_{n,s,\eta}$, we have
	$$T(x,y)=\Tr(yg(x)),$$
	where 
	\begin{equation}\label{eq:g}
		g(x)=b_0x^{q^{sk}}+b_1x^{q^{s(k-1)}}+(b_1x)^{q^{s(k+1)}}+\eta b_2x^{q^{s(k-2)}}+(\eta b_2x)^{q^{s(k+2)}}.
	\end{equation}
	It is easy to see that the rank of the bilinear form $T(x,y)$ equals to the rank of the $q$-polynomial $g$.

		Given the length and technical nature of the argument, we first outline the proof strategy for the reader's convenience. Our goal is to show that the rank of $g$ (and hence that of $T(x,y)$) is at least $n-2$. By Lemma \ref{lem:dickson_rank}, we only have to show that the associated Dickson matrix $D(g)$ is of rank at least $n-2$.
		\begin{itemize}
			\item\textbf{Step.\ 1}  We begin by establishing that it suffices to prove the theorem for \( s = 1 \) instead of every $s$; see Lemma \ref{lem:s->1_our}. 
		\end{itemize}
		The proof then splits into two cases based on the coefficient \( b_2 \).
		\begin{itemize}
			\item \textbf{Step.\ 2}  The case \( b_2 = 0 \) is resolved directly; see Lemma \ref{lem:2_2}.
		\end{itemize}
		\begin{itemize}
			\item \textbf{Step.\ 3} For the case $ b_2 \neq 0$, we have to handle the three cases \(k=3\), \(k=4\) and \(k=5\) separately, because for each  different value of $k$ we have to pick up at least two different \((n-2)\times(n-2)\) sub‑matrices of $D(g)$ to verify the required rank condition.
		\end{itemize}
		It should be noted that we could not find any apparent uniform pattern for choosing these sub‑matrices in Step 3; the choice depends heavily on the parameters. The computations for \(k=4\) and \(k=5\) are quite complicated and we have  to use symbolic computation software such as Maple.
		
		With this structure in mind, we now proceed to the detailed arguments.

	First we show that we only have to handle the proof of Theorem \ref{th:main} for $s=1$.
	\begin{theorem}\label{th:s->1}
		Let $s$ be an integer coprime to $n$ and let $\sigma$ be a generator of $\mathrm{Gal}(\mathbb{F}_{q^n}/\mathbb{F}_q)$. 
		For any $f_0,\dots,f_{n-1}\in \F_{q^n}$, define
		\begin{table*}[h]
			\begin{tabular}{lllll}
				$\varphi_1:$ & $\F_{q^{sn}} \rightarrow \F_{q^{sn}}$ &   and & $\varphi_2:$ & $\F_{q^{n}} \rightarrow \F_{q^{n}}$ \\
				& $x\mapsto \sum_{i=0}^{n-1} f_i x^{\sigma^{is}}$,&    &  & $x\mapsto \sum_{i=0}^{n-1} f_i x^{\sigma^{is}}$.
			\end{tabular}
		\end{table*}
		
		\noindent Then the rank of the $\F_{q^s}$-linear map $\varphi_1$ is the same as the rank of the $\F_q$-linear map $\varphi_2$.
	\end{theorem}
	Theorem \ref{th:s->1} and its proof can be found in \cite[Theorem 3.2]{neri_extending_2022}. The idea behind its proof and its applications to the construction of rank-metric codes have been appeared earlier which can be found in \cite{gow_galois_2009} and \cite{sheekey_new_2016}.
	
	\begin{lemma}\label{lem:s->1_our}
		Let $n$ and $s$ be two coprime positive integers. Let $\bar{q}=q^s$. Define an $\F_{\bar{q}}$-linear map
		\begin{align*}
			\psi:~ &\F_{\bar{q}^n} \rightarrow \F_{\bar{q}^n} \\
			&x \mapsto b_0x^{\bar{q}^{k}}+b_1x^{\bar{q}^{k-1}}+(b_1x)^{\bar{q}^{k+1}}+\eta b_2x^{\bar{q}^{k-2}}+(\eta b_2x)^{\bar{q}^{k+2}},
		\end{align*}
		where $\eta,b_0,b_1,b_2\in \F_{q^n}$. Then the $\F_q$-linear map $\psi|_{\F_{q^n}}$ has the same rank of the $\F_{\bar{q}}$-linear map $\psi$ over $\F_{q^{sn}}$. 
	\end{lemma}
	The proof of Lemma \ref{lem:s->1_our} is a direct consequence of Theorem \ref{th:s->1}. It tells us that we only have to prove Theorem \ref{th:main} for $s=1$ which will be handled in the rest part of this section.
	
	When $b_2=0$, the proof becomes quite simple.
	\begin{lemma}\label{lem:2_2}
		If $b_2=0$, then $\dim_{\F_q}\ker g(x)\leq2$.
	\end{lemma}
	\begin{proof}
		Now 
		\begin{align*}
			g(x) &=b_0x^{q^{k}}+b_1x^{q^{k-1}}+(b_1x)^{q^{k+1}}\\
			&= \left(  b_0^{q^{k+1}}x^q+b_1^{q^{k+1}}x +(b_1x)^{q^{2}}     \right)^{q^{k-1}}. 
		\end{align*}
		As the degree of $b_0^{q^{k+1}}x^q+b_1^{q^{k+1}}x +(b_1x)^{q^{2}}$ is at most $q^2$, it has at most $q^2$ distinct roots which implies $\dim_{\F_q}\ker g(x)\leq2$.
	\end{proof}
	
	For $b_2\neq 0$, we will finish the proof by investigating the Dickson matrix associated with $g$. Depending on the value of $k$, we separate the proof into three cases.
	
	\subsection{Proof of Theorem \ref{th:main} with $k=3$}\label{subsec:k=3}
	\begin{proof}
		Suppose to the contrary that there are $b_1\in \mathbb{F}_{q^6} $ and $b_0,b_2\in\mathbb{F}_{q^3}$ such that $b_2\neq0$ and  $\rank(D( g))<4$. 
		
		Now
		$$D( g):=
		\begin{pmatrix}
			0& \eta b_2 & b_1&b_0&  b_1^{q^4}&(\eta b_2)^{q^5}\\ 
			\eta b_2&0& (\eta b_2)^q & b_1^q&b_0^q&  b_1^{q^5}\\
			b_1&(\eta b_2)^{q}&0& (\eta b_2)^{q^2} &b_1^{q^2}&b_0^{q^2}\\
			b_0&b_1^{q}&(\eta b_2)^{q^2}&0& (\eta b_2)^{q^3} &b_1^{q^3}\\
			b_1^{q^4}&b_0^{q}&b_1^{q^2}&(\eta b_2)^{q^3}&0&(\eta b_2)^{q^4} \\
			(\eta b_2)^{q^5}&b_1^{q^5}&b_0^{q^2}&b_1^{q^3}&(\eta b_2)^{q^4}&0 \\
		\end{pmatrix}.$$ 
		By  $\rank(D( g))<4$, the determinant of any $4\times 4$ submatrix of $D( g)$ is $0$. 
		Let $M_1$ be leading principal $4\times 4$ submatrix of $D( g)$.
		Then 
		\begin{align}
			\nonumber
			\det (M_1)
			=&(b_0(\eta b_2)^q)^2+(\eta b_2)^{2(q^2+1)}+b_1^{2(q+1)}\\
			&-2\left(b_0(\eta b_2)^{q^2+q+1}+b_0(\eta b_2)^qb_1^{q+1}+(\eta b_2)^{q^2+1}b_1^{q+1}\right)=0.
			\label{eq_2_4}
		\end{align}
		Let $M_2$ be the $4\times 4$ submatrix in $D( g)$ consisting of elements from the columns/rows with indices in $\{1,2,4,5\}$.
		Then 	$$M_2=
		\begin{pmatrix}
			0& \eta b_2 &b_0&b_1^{q^4}\\
			\eta b_2&0& b_1^q&b_0^q\\
			b_0&b_1^q&0& (\eta b_2)^{q^3}\\
			b_1^{q^4}&b_0^{q}& (\eta b_2)^{q^3}&0
		\end{pmatrix}$$
		and
		\begin{align}
			\nonumber
			\det (M_2)=&b_0^{2(q+1)}+(\eta b_2)^{2(q^3+1)}+b_1^{2(q^4+q)}\\
			&-2\left(b_0^{q+1}(\eta b_2)^{q^3+1}+(\eta b_2)^{q^3+1}b_1^{q^4+q}+b_0^{q+1}b_1^{q^4+q}\right)=0. 		\label{eq_2_5}
		\end{align}
		Depending on the value of $b_0$ and $b_1$, we separate the rest part into two cases.
		
		\medskip
		
		\noindent\textbf{Case 1:} $b_0=0$.
		By \eqref{eq_2_4} and \eqref{eq_2_5}, we get $$(\eta b_2)^{q^2+1}=b_1^{q+1},$$ and $$(\eta b_2)^{q^3+1}=b_1^{q^4+q},$$
		respectively. 
		As $\eta b_2\neq 0$, from the two equations above we derive $b_1\neq 0$ and $(\eta b_2)^{(q^2+1)(q^3-q^2+q)}=(\eta b_2)^{q^3+1}$, which implies $$ (\eta b_2)^{(q^3-1)(q^2-q+1)}=1.$$
		Since $b_2\in\mathbb{F}_{q^3}^*$, we have $\eta ^{\frac{q^6-1}{q+1}}=1$ which contradicts the assumption that $\eta$ is a non-square in $\F_{q^6}$.
		\medskip
		
		\noindent\textbf{Case 2:} $b_0\neq0$. By \eqref{eq_2_4},  we get
		\begin{equation}\label{eq:square_k=3}
			\left(b_0(\eta b_2)^q+(\eta b_2)^{q^2+1}-b_1^{q+1}\right)^2=4b_0(\eta b_2^{q^2+q+1}).
		\end{equation}
		Obviously, the left-hand side of the above equation is a square.
		However, as $\eta$ is not a square in $\F_{q^6}$ and $b_2\in \F_{q^3}$, the right-hand side of \eqref{eq:square_k=3} is a non-square in $\F_{q^6}$ which leads to a contradiction.
	\end{proof}
	
	\subsection{Proof of Theorem \ref{th:main} with $k=4$}\label{subsec:k=4}
	\begin{proof}
		Suppose to the contrary that there are $b_1\in \fqn $ and $b_0,b_2\in\mathbb{F}_{q^4}$ such that $b_2\neq0$ and  $\rank(D( g))<6$ which means the determinant of any $i\times i$ submatrix of $D( g)$ is $0$, where $5<i\leq 8$. 
		
		Let $M_1$ be the $6\times 6$ submatrix in $D( g)$ consisting of elements from the columns/rows with indices in $\{1,2,3,5,6,7\}$.
		Then 	$$M_1=
		\begin{pmatrix}
			0&0& \eta b_2 &b_0&b_1^{q^5}&(\eta b_2)^{q^6}\\ 
			0&0&0 & b_1^q&b_0^{q}&  b_1^{q^6}\\
			\eta b_2&0&0& (\eta b_2)^{q^2}&b_1^{q^2}&b_0^{q^2} \\
			b_0&b_1^{q}&(\eta b_2)^{q^2}&0&0&	(\eta b_2)^{q^4} \\
			b_1^{q^5}&b_0^{q}&b_1^{q^2}&0&0&0\\
			(\eta b_2)^{q^6}&  b_1^{q^6}&b_0^{q^2}&	(\eta b_2)^{q^4}&0&0
		\end{pmatrix},$$
		and by calculation, we get
		\begin{align*}
			\det(M_1)=& -A_1^2+2A_1B_1+2A_1C_1+2A_1D_1+2A_1E_1-2A_1F_1-2A_1G_1\\
			&-B_1^2-2B_1C_1-2B_1D_1-2B_1E_1+2B_1F_1+2B_1G_1-C_1^2\\
			&+2C_1D_1-2C_1E_1-2C_1F_1-2C_1G_1-D_1^2-2D_1E_1+2D_1F_1\\
			&+2D_1G_1-E_1^2+2E_1F_1+2E_1G_1-F_1^2-2F_1G_1-G_1^2\\
			&+4b_1^{q^6+q^5+q^2+q}(\eta b_2)^{q^4+1},
		\end{align*}
		where $A_1=b_0^{q^2+q+1}$, $B_1=b_0b_1^{q^6+q^2}$, $C_1=b_0^q(\eta b_2)^{q^4+1}$, $D_1=b_0^q(\eta b_2)^{q^6+q^2}$, $E_1=b_0^{q^2}b_1^{q^5+q}$, $F_1=b_1^{q^2+q}(\eta b_2)^{q^6}$, $G_1=b_1^{q^5+q^6}(\eta b_2)^{q^2}$.
		
		By simplification and the assumption on the rank of $D( g)$, we get	
		\begin{align}
			\nonumber
			\det(M_1)=&4\left(b_1^{q^6+q^5+q^2+q}(\eta b_2)^{q^4+1}+C_1D_1-C_1F_1-C_1G_1\right)\\
			\label{eq:n_8}
			&-(A_1-B_1-C_1-D_1-E_1+F_1+G_1)^2=0,
		\end{align}

		Let $M_2$ be leading principal $6\times 6$ submatrix of $D( g)$.
		By calculation, we get
		\begin{align*}
			\det(M_2)=&-A_2^2-2A_2B_2+2A_2C_2-2A_2D_2+2A_2E_2\\
			&-B_2^2+2B_2C_2-2B_2D_2+2B_2E_2-C_2^2+2C_2D_2\\
			&-2C_2E_2-D_2^2-2D_2E_2-E_2^2+4A_2B_2D_2/C_2,
		\end{align*}
		where $A_2=b_0b_1^{q^2}(\eta b_2)^q$, $B_2=b_0^qb_1(\eta b_2)^{q^2}$, $C_2=(b_1)^{q^2+q+1}$, $D_2=b_1^q(\eta b_2)^{q^3+1}$, $E_2=b_1^{q^5}(\eta b_2)^{q^2+q}$. Here, $A_2B_2D_2/C_2$ should be interpreted as $b_0^{1+q}(\eta b_2)^{1+q+q^2+q^3}$ which actually does not depend on the value of $b_1$.
		
		Furthermore, by simplification and the assumption on the rank of $D( g)$, we obtain
		\begin{align}
			\nonumber
			\det(M_2)=&4\left(A_2B_2D_2/C_2-A_2D_2-B_2D_2+C_2D_2\right)\\
			\label{eq:n_8,1}
			&-(A_2+B_2-C_2-D_2-E_2)^2=0.
		\end{align}

		Depending on the value of $b_0$, we separate the rest part into two cases.
		
		\medskip
		
		\noindent\textbf{Case 1:} $b_0=0$.
		Let $w$ denote a primitive element in $\fqn$.
		
		If $b_1= 0$, then 
		\[ g(x)=\eta b_2x^{q^{s(k-2)}}+(\eta b_2x)^{q^{s(k+2)}},\]
		which implies
		\[\det(D( g))=\left((\eta b_2)^{q^4+1}-(\eta b_2)^{q^2(q^4+1)}\right)^{2(q+1)}=0.\]
		
		Hence $(\eta b_2)^{(q^2-1)(q^4+1)}=1$, which implies $\eta b_2=w^{(q^2+1)c_0}$ for some integer $c_0>0$.
		By the assumption that $\eta$ is a non-square, we get a contradiction.
		
		If $b_1\neq 0$,	by \eqref{eq:n_8} and $b_0=0$, we get \begin{align}\label{eq:n_8_0}
			\det(M_1)=4b_1^{q^6+q^5+q^2+q}(\eta b_2)^{q^4+1}-\left(b_1^{q^2+q}(\eta b_2)^{q^6}+b_1^{q^6+q^5}(\eta b_2)^{q^2}\right)^2=0.
		\end{align}
		Through direct analysis we get
		\begin{equation}\label{eq:direct}
			b_1^{q^2+q}(\eta b_2)^{q^6}+b_1^{q^6+q^5}(\eta b_2)^{q^2}\in\F_{q^4}.
		\end{equation}
		By looking at the powers, it is easy to see that
		$$4b_1^{q^6+q^5+q^2+q}(\eta b_2)^{q^4+1}=w^{4c_1+2},$$
		for some $0<c_1<q^n/4$.
		Together with \eqref{eq:n_8_0}, we obtain $b_1^{q^2+q}(\eta b_2)^{q^6}+b_1^{q^6+q^5}(\eta b_2)^{q^2}$ is not a square in $\F_{q^8}$ which contradicts \eqref{eq:direct}. 
		
		\medskip
		
		\noindent\textbf{Case 2:} $b_0\neq0$.
		
		Let $Y=b_1^{q+1}-b_0(\eta b_2)^{q}$.
		Let $U=A_1-B_1-C_1-D_1-E_1+F_1+G_1$.
		By  \eqref{eq:n_8} and $b_0\in\F_{q^4}$, we get 
		\begin{equation}\label{eq:1_1}
			U^2=4(\eta b_2)^{q^4+1}(b_1^{q+1}-b_0(\eta b_2)^{q})^{q^5+q}=	4(\eta b_2)^{q^4+1}Y^{q^5+q},
		\end{equation}
		and $U\in\F_{q^4}$.
		
		Let $W=A_2+B_2-C_2-D_2-E_2$. By \eqref{eq:n_8,1}, we get
		\begin{equation}\label{eq:1_2}
			W^2=4(\eta b_2)^{q^3+1}(b_1^{q+1}-b_0(\eta b_2)^{q})^{q+1}=4(\eta b_2)^{q^3+1}Y^{q+1}.
		\end{equation}
		
		As $\eta$ is not a square in $\fqn$ and $b_0,b_2\in\F_{q^4}^*$,  $b_0(\eta b_2)^{q} $ is also not a square in $\F_{q^8}$. It is clear that $b_1^{q+1}$ is a square in $\F_{q^8}$. Hence $b_1^{q+1}$ is not equal to $b_0(\eta b_2)^{q} $ and $Y\neq 0$. 
		Then $W\neq 0$ and $U\neq 0$. 
		By \eqref{eq:1_1} and \eqref{eq:1_2}, we obtain
		\begin{equation}	\label{eq:2_0}
			\frac{W^{q^4+1}}{U^{q^3+1}}=\frac{	4(\eta b_2)^{(q^3+1)(q^4+1)/2}Y^{(q+1)(q^4+1)/2}}{	4(\eta b_2)^{(q^4+1)(q^3+1)/2}Y^{q(q^4+1)(q^3+1)/2}}=1.
		\end{equation}
		
		Let $w$ be a primitive element in $\F_{q^8}$. By  $U\in\F_{q^4}^*$, we get \(U=w^{k_1(q^4+1)}\) for some integer $k_1$. Moreover, by \eqref{eq:2_0} , we get \(W=w^{k_2(q+1)}\) and 
		\begin{equation}\label{eq:k_1,2}
			k_1(q^2-q+1)\equiv k_2\mod(q^2+1)(q-1),
		\end{equation} for some integer $k_2$.
		
		As $\eta$ is not a square in $\F_{q^8}$ and $b_2\in\F_{q^4}$, there exists an odd integer $k_0$ such that $\eta b_2=w^{k_0}$. By \eqref{eq:1_1}, we obtain 
		\begin{equation}\label{eq:2_1}
			2Y=w^{2k_1q^3-k_0q^3}w^{k_3(q^4-1)},
		\end{equation}
		for some integer $k_3$.
		By \eqref{eq:1_2}, we obtain 
		\begin{equation}\label{eq:2_2}
			2Y=w^{2k_2-k_0(q^2-q+1)}w^{k_4(q^4+1)(q^2+1)(q-1)},
		\end{equation}
		for some integer $k_4$.
		By \eqref{eq:2_1} and \eqref{eq:2_2}, we get 
		$$2k_1q^3-k_0q^3+k_3(q^4-1)=2k_2-k_0(q^2-q+1)+k_4(q^4+1)(q^2+1)(q-1)+i(q^8-1),$$
		for some integer $i$. Then  
		\begin{equation}\label{eq:2_3}
			\left(-k_0+k_3(q+1)-k_4(q^4+1)-i(q^4+1)(q+1)\right)(q^2+1)(q-1)+2(k_1q^3-k_2)=0,
		\end{equation}
		By \eqref{eq:k_1,2},  $k_1q^3-k_2\equiv k_1q^3- k_1(q^2-q+1)\equiv k_1(q^2+1)(q-1) \equiv 0 \pmod{(q^2+1)(q-1)}$. Thus there is an integer $j$ such that
		\[
		k_1q^3-k_2=j(q^2+1)(q-1).
		\]
		
		Hence \eqref{eq:2_3} implies 
		\[2j+k_3(q+1)-k_4(q^4+1)-i(q^4+1)(q+1)=k_0.\]
		Then the left-hand side of the equation above is even and its right-hand side $k_0$ is odd which leads to a contradiction.
	\end{proof}

	\subsection{Proof of Theorem \ref{th:main} with $k=5$}\label{subsec:k=5}
	\begin{proof}
		Suppose to the contrary that there are $b_1\in \fqn $ and $b_0,b_2\in\mathbb{F}_{q^4}$ such that $b_2\neq0$ and  $\rank(D( g))<8$. The determinant of any 	 $i\times i$ submatrix of $D( g)$ is $0$, where $7<i\leq 10$. 
		
		Let $M_1$ be the leading principal $8\times 8$ submatrix of $D( g)$. 	By calculation, we get
		\begin{align*}
			&\det(M_1)\\
			=& A_1^2-2A_1B_1-2A_1C_1-2A_1D_1-2A_1E_1+2A_1F_1+2A_1G_1-2A_1H_1+2A_1I_1-2A_1J_1\\
			&+B_1^2+2B_1C_1-2B_1D_1+2B_1E_1-2B_1F_1-2B_1G_1+2B_1H_1-2B_1I_1+2B_1J_1+C_1^2\\
			&-2C_1D_1+2C_1E_1-2C_1F_1-2C_1G_1+2C_1H_1-2C_1I_1+2C_1J_1+D_1^2-2D_1E_1+2D_1F_1\\
			&+2D_1G_1-2D_1H_1+2D_1I_1+2D_1J_1+E_1^2-2E_1F_1-2E_1G_1+2E_1H_1-2E_1I_1+2E_1J_1+F_1^2\\
			&+2F_1G_1-2F_1H_1+2F_1I_1-2F_1J_1+G_1^2-2G_1H_1+2G_1I_1-2G_1J_1+H_1^2-2H_1I_1-2H_1J_1\\
			&+I_1^2-2I_1J_1+Jv^2+4(\eta b_2)^{q^4+q^3+q^2+q+1}(b_1^{q^7+q^6}(\eta b_2)^{q^2}-b_0b_1^{q^7+q^2}-b_0^{q^2}b_1^{q^6+q^1}),
		\end{align*}
		where $A_1=b_0^{q^2+1}(\eta b_2)^{q^3+q}$, $B_1=b_0b_1^{q^3+q^2}(\eta b_2)^{q}$, $C_1=b_0^qb_1^{q^3+1}(\eta b_2)^{q^2}$, $D_1=b_0^q(\eta b_2)^{q^4+q^2+1}$, $E_1=b_0^{q^2}b_1^{q+1}(\eta b_2)^{q^3}$, $F_1=b_1^{q^3+q^2+q+1}$, $G_1=b_1^{q^7+1}(\eta b_2)^{q^2+q^3}$, $H_1=b_1^{q+q^2}(\eta b_2)^{q^4+1}$, $I_1=b_1^{q^6+q^3}(\eta b_2)^{q^2+q}$, $J_1=(\eta b_2)^{q+q^2+q^3+q^7}$. It can be further simplified as
		\begin{align}
			\nonumber
			\det(M_1) = &(A_1-B_1-C_1+D_1-E_1+F_1+G_1-H_1+I_1-J_1)^2\\
			\label{eq:n_10}
			&-4(A_1D_1-D_1J_1+H_1J_1)\\
			\nonumber
			&-4(\eta b_2)^{q^4+q^3+q^2+q+1}\left(b_1^{q^7+q^6}(\eta b_2)^{q^2}
			-b_0b_1^{q^7+q^2}-b_0^{q^2}b_1^{q^6+q^1}\right)=0	
		\end{align}
		
		By \eqref{eq:n_10}, we get 
		\begin{equation}\label{eq:square_10}
			\begin{aligned}
				&(A_1-B_1-C_1+D_1-E_1+F_1+G_1-H_1+I_1-J_1)^2\\
				=&4(\eta b_2)^{q^4+q^3+q^2+q+1}\Delta,
			\end{aligned}
		\end{equation}
		where
		\[
		\Delta = b_1^{q^7+q^6}(\eta b_2)^{q^2}-b_0b_1^{q^7+q^2}-b_0^{q^2}b_1^{q^6+q}+b_0^{q^2+q+1}-b_0^q(\eta b_2)^{q^2+q^7}+b_1^{q^2+q}(\eta b_2)^{q^7}.
		\]

		Depending on the value of $\Delta$, we continue with the proof in two different cases.
		
		\medskip
		
		\noindent\textbf{Case 1:}  $\Delta\neq 0$.	Obviously, the left-hand side of \eqref{eq:square_10} is a square in $\F_{q^{10}}$.
		By $b_0\in\mathbb{F}_{q^5}$, we get \(4\Delta\in\mathbb{F}_{q^5}^*\), which is also a square in $\F_{q^{10}}$.
		
		As $\eta$ is not a square in $\fqn$, by $b_2\in\mathbb{F}_{q^5}^*$ and $q$ is odd, we obtain that \((\eta b_2)^{q^4+q^3+q^2+q+1}\) is a non-square which leads to a contradiction to \eqref{eq:square_10}.
		
		\medskip
		
		\noindent\textbf{Case 2:}  $\Delta=0$. Now
		by \eqref{eq:square_10} we have 
		\begin{equation}\label{eq:equal_0}
			A_1-B_1-C_1+D_1-E_1+F_1+G_1-H_1+I_1-J_1=0,
		\end{equation}
		which implies 
		
		\begin{equation}\label{eq_10_2^*}
			\begin{aligned}
				&(A_1-B_1-C_1-D_1-E_1+F_1+G_1+H_1+I_1-J_1)^2\\
				=&4(H_1-D_1)(A_1-B_1-C_1-E_1+F_1+G_1+I_1-J_1)\\
				=&4(\eta b_2)^{q^4+1}\left(b_1^{q^2+q}-b_0^{q}(\eta b_2)^{q^2}\right)\left(b_0^{q^2+1}(\eta b_2)^{q^3+q}-b_0b_1^{q^3+q^2}(\eta b_2)^{q}+b_1^{q^6+q^3}(\eta b_2)^{q^2+q}-(\eta b_2)^{q+q^2+q^3+q^7}\right)\\
				+&4(\eta b_2)^{q^4+1}b_1\left((\eta b_2)^{q^2}-b_1^{q^2+q}\right)U^q
				\\
				=&4(\eta b_2)^{q^4+1}\left(b_0^{q+1}b_1^{q^3+q^2}(\eta b_2)^{q^2+q}+b_0^{q^2+1}b_1^{q^2+q}(\eta b_2)^{q^3+q}-b_0b_1^{q^3+2q^2+q}(\eta b_2)^q+b_1^{q^6+q^3+q^2+q}(\eta b_2)^{q^2+q}\right)\\
				-&4(\eta b_2)^{q^4+q^2+q+1}\left(b_0^{q^2+q+1}(\eta b_2)^{q^3}-b_0^q(\eta b_2)^{q^7+q^3+q^2}+b_1^{q^2+q}(\eta b_2)^{q^7+q^3}+b_0^qb_1^{q^6+q^3}(\eta b_2)^{q^2}\right)\\
				+&4(\eta b_2)^{q^4+1}\left(b_1(\eta b_2)^{q^2}-b_1^{q^2+q+1}\right)U^q,
			\end{aligned}
		\end{equation}
		where 
		\begin{equation}\label{eq:def_U}
			U=b_0b_1^{q^2}(\eta b_2)^{q}+b_0^{q}b_1(\eta b_2)^{q^2}-b_1^{q^2+q+1}-b_1^{q^6}(\eta b_2)^{q^2+q}.
		\end{equation}
		By $\Delta=0$, we get 
		\[b_0^{q^2+q+1}(\eta b_2)^{q^3}-b_0^q(\eta b_2)^{q^7+q^3+q^2}+b_1^{q^2+q}(\eta b_2)^{q^7+q^3}=(\eta b_2)^{q^3}\left(b_0b_1^{q^7+q^2}+b_0^{q^2}b_1^{q^6+q}-b_1^{q^7+q^6}(\eta b_2)^{q^2}\right).\]
		Together with \eqref{eq_10_2^*}, then 
		\begin{equation}\label{eq_10_2}
			\begin{aligned}
				&(A_1-B_1-C_1-D_1-E_1+F_1+G_1+H_1+I_1-J_1)^2\\
				=&4(\eta b_2)^{q^4+q+1}\left(b_0b_1^{q^2}(b_0^qb_1^{q^3}(\eta b_2)^{q^2}+b_0^{q^2}b_1^q(\eta b_2)^{q^3})-b_1^{q^3+q^2+q}(b_0b_1^{q^2}-b_1^{q^6}(\eta b_2)^{q^2})\right)\\
				-&4(\eta b_2)^{q^4+q^2+q+1}\left((\eta b_2)^{q^3}(b_0b_1^{q^7+q^2}+b_0^{q^2}b_1^{q^6+q}-b_1^{q^7+q^6}(\eta b_2)^{q^2})+b_0^qb_1^{q^6+q^3}(\eta b_2)^{q^2}\right)\\
				+&4(\eta b_2)^{q^4+1}(b_1(\eta b_2)^{q^2}-b_1^{q^2+q+1})U^q\\
				=&4(\eta b_2)^{q^4+1}U^{q+1},
			\end{aligned}
		\end{equation}
		where $	U=b_0b_1^{q^2}(\eta b_2)^{q}+b_0^{q}b_1(\eta b_2)^{q^2}-b_1^{q^2+q+1}-b_1^{q^6}(\eta b_2)^{q^2+q}.$
		
		Depending on the value of $b_0$ and $b_1$, we consider two cases.
		\medskip
		
		\textbf{Case 2.1:} $b_0=b_1=0$.
		
		Let \(M_2\) denote the \((n - 2)\times(n - 2)\) matrix obtained from \(D( g)\) after removing its first \(2\) columns and last \(2\) rows. Then $$M_2=
		\begin{pmatrix}
			0& \eta b_2 &0&0& 0&(\eta b_2)^{q^7}&0&0\\ 
			0&0&(\eta b_2)^q &0&0& 0&(\eta b_2)^{q^8}&0\\
			0&0&0& (\eta b_2)^{q^2} &0&0&0&(\eta b_2)^{q^9}\\
			0&0&0&0& (\eta b_2)^{q^3} &0&0&0\\
			
			0&0&0&0&0& (\eta b_2)^{q^4} &0&0\\
			
			(\eta b_2)^{q^2}&0&0&0&0&0&(\eta b_2)^{q^5}&0 \\
			0&(\eta b_2)^{q^3}&0&0&0&0&0&(\eta b_2)^{q^6} \\
			0&0&(\eta b_2)^{q^4}&0&0&0&0&0
		\end{pmatrix}.$$ 
		By calculation, we get 
		$$\det(M_2)=-(\eta b_2)^{1+2q^2+q^3+2q^4+q^6+q^8}=0,$$  which contradicts $\eta b_2\neq 0$.
		\medskip
		
		\textbf{Case 2.2:} At least one of \(b_0\) and \(b_1\) is non-zero. 
		
		It is straightforward to show that $U\neq 0$; otherwise, by \eqref {eq_10_2}, we get
		\[
		A_1-B_1-C_1-D_1-E_1+F_1+G_1+H_1+I_1-J_1=0
		\]
		from which together with \eqref{eq:equal_0} we derive $D_1=H_1$, i.e.,
		$$b_0^{q}(\eta b_2)^{q^2}=b_1^{q^2+q},$$
		which contradicts the assumption that $\eta$ is a non-square in $\F_{q^{10}}$.
		
		Let $M_2$ be the $8\times 8$ submatrix in $D( g)$ consisting of the elements in  columns/rows with indices in $\{1,2,3,4,6,7,8,9\}$. Then
		$$	M_3=
		\begin{pmatrix}
			0&0&0& \eta b_2 &b_0&  b_1^{q^6}&(\eta b_2)^{q^7}&0\\ 
			0&0&0&0 & b_1^q&b_0^q&  b_1^{q^7}&(\eta b_2)^{q^8}\\
			0&0&0&0& (\eta b_2)^{q^2} &b_1^{q^2}&b_0^{q^2}&b_1^{q^8}\\
			\eta b_2&0&0&0&0& (\eta b_2)^{q^3} &b_1^{q^3}&b_0^{q^3}\\
			
			b_0&b_1^{q}&(\eta b_2)^{q^2}&0&0&0&0&(\eta b_2)^{q^5} \\
			b_1^{q^6}&b_0^{q}&b_1^{q^2}&(\eta b_2)^{q^3}&0&0&0&0\\
			(\eta b_2)^{q^7} &b_1^{q^7}&b_0^{q^2}&b_1^{q^3}&0&0&0&0\\
			0&(\eta b_2)^{q^8} &b_1^{q^8}&b_0^{q^3}&(\eta b_2)^{q^5}&0&0&0
		\end{pmatrix}.$$ 
		By calculation, we get
		\begin{equation}\label{eq_10_4}
			\begin{aligned}
				& \det(M_3)\\
				=&(A_2-B_2-C_2-D_2+E_2+F_2-G_2-H_2-I_2+J_2+K_1+K_2-K_3+K_4-K_5+K_6)^2\\
				&-4(\eta b_2)^{q^5+1}\left(b_0^qb_1^{q^3}(\eta b_2)^{q^2}+b_0^{q^2}b_1^{q}(\eta b_2)^{q^3}-b_1^{q^3+q^2+q}-b_1^{q^7}(\eta b_2)^{q^3+q^2}\right)^{q^5+1},
			\end{aligned}	
		\end{equation}
		where $A_2=b_0^{q^3+q^2+q+1}$, $B_2=b_0^{q+1}b_1^{q^8+q^3}$, $C_2=b_0^{q^3+1}(\eta b_2)^{q^8+q^3}$, $D_2=b_0^{q^3+1}b_1^{q^7+q^2}$,  $E_2=b_0b_1^{q^3+q^2}(\eta b_2)^{q^8}$, $F_2=b_0b_1^{q^8+q^7}(\eta b_2)^{q^3}$, $G_2=b_0^{q^2+q}(\eta b_2)^{q^5+1}$,
		$H_2=b_0^{q^3+q}(\eta b_2)^{q^7+q^2}$,
		$I_2=b_0^{q^3+q^2}b_1^{q^6+q}$,
		$J_2=b_0^{q^3}b_1^{q^2+q}(\eta b_2)^{q^7}$,
		$K_1=b_0^{q^3}b_1^{q^7+q^6}(\eta b_2)^{q^2}$,
		$K_2=b_1^{q^8+q^6+q^3+q}$, $K_3=b_1^{q^8+q}(\eta b_2)^{q^7+q^3}$, $K_4=b_1^{q^7+q^2}(\eta b_2)^{q^5+1}$, $K_5=b_1^{q^6+q^3}(\eta b_2)^{q^8+q^2}$, $K_6=(\eta b_2)^{q^8+q^7+q^3+q^2}$.
		
		Set 
		$$W=A_2-B_2-C_2-D_2+E_2+F_2-G_2-H_2-I_2+J_2+K_1+K_2-K_3+K_4-K_5+K_6.$$
		As $\det(M_3)=0$, from by \eqref{eq_10_4} it follows that
		\begin{equation}\label{eq_10_6}
			W^2 = 4(\eta b_2)^{q^5+1}U^{q^6+q},
		\end{equation}
		in which $U$ is defined by \eqref{eq:def_U}.
		
		On the other hand, let $V=A_1-B_1-C_1-D_1-E_1+F_1+G_1+H_1+I_1-J_1$.
		By \eqref{eq_10_2}, we get 
		\begin{equation}
			\label{eq_10_5}
			V^2=4(\eta b_2)^{q^4+1}U^{q+1}.
		\end{equation}
		
		Together with $U\neq 0$ and \eqref{eq_10_6}, we get $V\neq 0$ and $W\in\mathbb{F}_{q^5}^*$.
		By \eqref{eq_10_5} and \eqref{eq_10_6}, we obtain
		\begin{equation}\label{eq_10_7}
			\frac{W^{q^4+1}}{V^{q^5+1}}=\frac{4(\eta b_2)^{(q^5+1)\frac{q^4+1}{2}}U^{q(q^5+1)\frac{q^4+1}{2}}}{4(\eta b_2)^{(q^4+1)\frac{q^5+1}{2}}U^{(q+1)\frac{q^5+1}{2}}}=1.
		\end{equation} 
		Let $w$ be a primitive element in $\F_{q^{10}}$. By $W\in\F_{q^5}^*$, we can set \(V=w^{k_1}\) and \(W=w^{k_2(q^5+1)}\) for some integers $k_1$ and $k_2$. Moreover, by \eqref{eq_10_7},
		\begin{equation}\label{eq_10_8}
			k_1 \equiv k_2q^4(q+1)\mod(q^5-1).
		\end{equation}
		As $\eta$ is not a square in $\F_{q^{10}}$ and $b_2\in\F_{q^5}$, there exists an odd integer $k_0$ such that $\eta b_2=w^{k_0}$. By \eqref{eq_10_5}, we obtain 
		\begin{equation}\label{eq_10_9}
			2U=w^{\frac{2k_1-k_0(q^4+1)}{q+1}}w^{\frac{k_3(q^{10}-1)}{q+1}},
		\end{equation}
		for some integer $k_3$.
		By \eqref{eq_10_6}, we obtain 
		\begin{equation}\label{eq_10_10}
			2U=w^{2k_2q^4-k_0q^4}w^{k_4(q^5-1)},
		\end{equation}
		for some integer $k_4$.
		By \eqref{eq_10_9} and \eqref{eq_10_10}, we get 
		$$2k_1-k_0(q^4+1)+k_3(q^{10}-1)=(q+1)\left(2k_2q^4-k_0q^4+k_4(q^5-1)+i(q^{10}-1)\right),$$
		for some integer $i$. Then  
		\begin{equation}\label{eq_10_11}
			\left(k_0-k_4(q+1)+(q^5+1)(k_3-i(q+1))\right)(q^5-1)+2(k_1-k_2q^4(q+1))=0,
		\end{equation}
		By \eqref{eq_10_8}, there is an integer $j$ such that $k_1-k_2q^4(q+1)=j(q^5-1)$.
		Hence \eqref{eq_10_11} implies 
		\[-k_4(q+1)+(q^5+1)(k_3-i(q+1))+2j=-k_0.\]
		As $q$ is odd,  the left-hand side of the equation above is even which contradicts to $2\nmid k_0$.
	\end{proof}
	
	\section{Equivalence}\label{sec:equi}
	This section consists of two main results. First, we show that the $(n-2)$-codes constructed in Theorem \ref{th:main} are new. Second, we completely determine the equivalence between different members in the family constructed in  Theorem \ref{th:main}.
	
	For even $n>4$ and $d=n-2$, according to the summary in Section \ref{sec:intro}, there is only one known construction of maximum symmetric $d$-codes defined in \eqref{eq:S_1}. Therefore, we only have to prove the following theorem to show that $\cT_{n,t,\eta}$ defined in Theorem \ref{th:main} is new.
	
	\begin{theorem}\label{th:ineq}
		Let $k$ be an integer larger than $2$ and $n=2k$. Let $q$ be an odd prime power and $\eta$ a non-square element in $\fqn$.
		For any $s,t$ satisfying $\gcd(s,n)=\gcd(t,n)=1$,
		$\cS_{n,n-2,s}$ is inequivalent to $\cT_{n,t,\eta}$.
	\end{theorem}
	\begin{proof}
		For any element $T(x,y) \in \cT_{n,t,\eta}$, we have
		\begin{align*}
			T(x,y)&=\Tr\left(b_0x^{q^{k}}y+b_1\left(x^{q^{t(k-1)}}y+y^{q^{t(k-1)}}x\right)+\eta b_2\left(x^{q^{ t(k-2)}}y+y^{q^{t(k-2)}}x\right)\right)\\
			&=\Tr(yg(x)),
		\end{align*}
		for some $b_0,b_2\in \F_{q^k}$ and $b_1\in \F_{q^{2k}}$, 
		where 
		\begin{equation}\label{eq:def_g}
			g(x)=b_0x^{q^{tk}}+b_1x^{q^{t(k-1)}}+(b_1x)^{q^{t(k+1)}}+\eta b_2x^{q^{t(k-2)}}+(\eta b_2x)^{q^{t(k+2)}}.
		\end{equation}
		
		For any element $S(x,y)\in \cS_{n,n-2,s}$, we have
		\[\Tr\left(c_0xy+c_1\left(x^{q^{s}}y+y^{q^{s}}x\right)\right)=\Tr\left(yh(x)\right),\]
		where $h(x)=c_0x+c_1x^{q^{s}}+(c_1x)^{q^{n-s}}$.
		
		Suppose to the contrary that $\cT_{n,t,\eta}$ and $\cS_{n,n-2,s}$ are equivalent. As they are both $\F_q$-linear, we may assume that $a=1$ in \eqref{eq:phi}.
		Suppose that $q=p^m$ where $p$ is a prime number. 
		By definition, there exists a
		$q$-polynomial $f=\sum_{i=0}^{n-1}a_i X^{q^i}$ which defines a permutation map on $\fqn$ and an integer $\ell\in\{0,\dots,m-1\}$ such that
		\[
		\Tr\left(c_0^{p^\ell}f(x)f(y)+c_1^{p^\ell}\left(f(x)^{q^{s}}f(y)+f(y)^{q^{s}}f(x)\right)\right)\in \cT_{n,t,\eta},
		\]
		i.e.,
		\[
		\Tr\left(y\hat{f}\left(h^{p^\ell} (f(x))\right)\right)\in \cT_{n,t,\eta},
		\]
		for any $c_0,c_1\in \fqn$, where $h^{p^\ell}(x) =c_0^{p^\ell}x+c_1^{p^\ell}x^{q^{s}}+(c_1^{p^\ell}x)^{q^{n-s}}$.
		
		We only have to consider whether $\hat{f}\left(h^{p^\ell} (f(x))\right)$ equals some $g(x)$ which is defined in \eqref{eq:def_g}.
		
		Set $c_1=0$ in $h$. By computation, the coefficient of $x$ in  $\hat{f}\left(h^{p^\ell} (f(x))\right)$ is
		\[\sum_{i=0}^{n-1}(c_0^{p^\ell})^{q^{n-i}}a_i^{2q^{n-i}}.\]
		Since the coefficient of the term of $q$-degree $0$ in $g$ is zero, we obtain
		$$a_0^2x + \sum_{i=1}^{n-1}a_i^{2q^{n-i}}x^{q^{n-i}} = 0,$$
		for any $x\in\fqn$. By the one-to-one correspondence between $q$-polynomials of $q$-degree less than $n$ and $\F_q$-linear transformation over $\fqn$, we get
		\begin{equation}\label{eq_2_3}
			a_i=0, \text{ for each }i=0,1,\cdots,n-1.
		\end{equation}
		Hence $f$ is the null polynomial which contradicts the assumption on $f$.
		
		Therefore, $\cS_{n,n-2,s}$ is not equivalent to $\cT_{n,t,\eta}$.
	\end{proof}

		\begin{remark}
			According to \eqref{eq:dual}, the dual of \(\mathcal{T}_{n,s,\eta}\) in Theorem \ref{th:main} is
			\begin{align*}
				\mathcal{T}_{n,s,\eta}^\perp& = 
				\left\{ \mathrm{Tr}\left(a_0xy+\sum_{i=1}^{k-2}a_i\left(x^{q^{si}}y+y^{q^{si}}x\right)\right): 
				a_0,\dots,a_{k-2}\in\mathbb{F}_{q^n},\;  \right.\\ 	&\left.	\eta a_{k-2}+(\eta a_{k-2})^{q^k}=0 \right\}.
			\end{align*}				
			It is easy to see that there are exactly $q^{n(n-3)/2}$ elements in the above set.  No matter whether $d$ is odd or even, the size of the set cannot meet the upper bound in Theorem \ref{the1.2}.
			
			Furthermore, we observe that $\mathcal{T}_{n,t,\eta}$ cannot be obtained as the dual of any known maximum $2$-code.
	\end{remark}
	
	Next we look at the equivalence between different members of the family stated in Theorem \ref{th:main}. In fact, as Theorem \ref{th:ineq}, we can completely determine their equivalence not only for $n=6,8,10$ but also for any even $n\geq 6$.
	\begin{theorem}\label{th:equi_members}
		For any positive integer $k>2$, let $n=2k$. Let $p$ be an odd prime and $m\in \mathbb{Z}^+$. Set $q=p^m$. For any non-squares $\eta_1,\eta_2\in \F_q$ and any integers $s_1,s_2$ satisfying $0<s_1,s_2<2k$ and $\gcd(s_1,n)=\gcd(s_2,n)=1$, $\cT_{n,s_1, \eta_1}$ and $\cT_{n,s_2, \eta_2}$ are equivalent if and only if
		one of the following collections of conditions is satisfied:
		\begin{itemize}
			\item[(a)]
			$s_1\equiv s_2\mod n$, and there are $a\in\fqn$, $i\in\{0,1,\cdots,n-1\}$ and $r \in \{0,\dots,m-1\}$ such that   $\eta_2^{q^{s_1i}}=a^{1+q^{s_1(k-2)}}\eta_1^{p^r}$;
			\item[(b)] 
			$s_1\equiv- s_2\pmod{n}$, and there are $a\in\fqn$, $i\in\{0,1,\cdots,n-1\}$ and $r \in \{0,\dots,m-1\}$ such that  $\eta_2^{q^{s_1i}}=a^{1+q^{s_1(k+2)}}\eta_1^{p^r q^{s_1(k+2)}}$.
		\end{itemize} 
	\end{theorem}
	\begin{proof}
		Assume that $s_1\equiv s_2 e\pmod{n}$ for integer $e$ satisfying $1\leq e<n$. As $n$ is even, $e$ must be an odd integer and $ek\equiv k \pmod{n}$.
		
		Assume that $\cT_{n,s_1, \eta_1}$ and $\cT_{n,s_2, \eta_2}$ are equivalent. 	As in the proof of Theorem \ref{th:ineq}, for $i=1,2$, any element $T(x,y) \in \cT_{n,s_i,\eta_i}$ can be written as
		$$T(x,y)=\Tr(yg_i(x)),$$
		where 
		\[g_i(x):=b_0x^{q^{s_i k}}+b_1x^{q^{s_i (k-1)}}+(b_1x)^{q^{s_i (k+1)}}+\eta_i b_2x^{q^{s_i (k-2)}}+(\eta_i  b_2x)^{q^{s_i (k+2)}}.\]
		
		Thus there exist permutation $q$-polynomial $f=\sum_{i=0}^{n-1}a_iX^{q^{s_1i}}\in\fqn[X]$ and $\sigma$ defined by $\sigma(x)=x^{p^r}$ with $r\in \{0,\dots,m-1\}$ such that	
		\[
		\Tr\left(y\hat{f}\left(g_1^{\sigma} (f(x))\right)\right)\in \cT_{n,s_2, \eta_2}
		\]
		for any $b_0,b_2\in \F_{q^k}$ and $b_1\in \F_{q^n}$, where 
		\[
		g_1^{\sigma}(x)=b_0^\sigma x^{q^{s_1 k}}+b_1^\sigma x^{q^{s_1 (k-1)}}+(b_1^\sigma x)^{q^{s_1 (k+1)}}+\eta_1^\sigma b_2^\sigma x^{q^{s_1 (k-2)}}+(\eta_1  b_2^\sigma x)^{q^{s_1 (k+2)}}.
		\]
		
		As $b_0^\sigma, b_2^\sigma\in \F_{q^k}$ and $b_1^\sigma\in \F_{q^n}$, we may still write $b_i^\sigma$ as $b_i$ in the equation above by abuse of notation.
		
		It follows that 
		\begin{equation}\label{eq:f(g1)f=g_2}
			\hat{f}\circ \tilde{g}_1 \circ f \equiv  g_2 \pmod{X^{q^n}-X}
		\end{equation}
		for some $g_2=b_0'X^{q^{s_2 k}}+b_1'X^{q^{s_2 (k-1)}}+(b_1'X)^{q^{s_2 (k+1)}}+\eta_2 b_2'X^{q^{s_2 (k-2)}}+(\eta_2  b_2'X)^{q^{s_2 (k+2)}}$ where
		\[\tilde{g}_1=b_0X^{q^{s_1k}}+b_1X^{q^{s_1(k-1)}}+(b_1X)^{q^{s_1(k+1)}}+\eta_1^\sigma b_2X^{q^{s_1(k-2)}}+(\eta_1^\sigma b_2X)^{q^{s_1(k+2)}}.\]
		
		By expanding the left-hand side of \eqref{eq:f(g1)f=g_2}, we get
		\begin{equation}\label{equi:0}
			\footnotesize
			\sum_{j = 0}^{n - 1} \left( \sum_{i = 0}^{n - 1} \left( b_0^{q^{s_1(k - i)}} a_i^{q^{s_1(n - i)}} a_{j + i+k}^{q^{s_1(k - i)}}  + \sum_{r = 1}^{2} \left( c_r^{q^{s_1(k - i +r)}} a_i^{q^{s_1(n - i)}} a_{j + i + k-r}^{q^{s_1(k - i+r)}} + c_r^{q^{s_1(n - i)}} a_i^{q^{s_1(n - i)}} a_{j + i - k+r}^{q^{s_1(k-r - i)}} \right) \right)\right) X^{q^{s_1j}},
		\end{equation}
		where we set $c_1=b_1$ and $c_2=\eta_1^\sigma b_2$ for convenience. Note that the subscripts of $a_i$'s are computed modulo $n$.
		
		By comparing the coefficients of the term $X^{q^{s_1j}}$ on both side of \eqref{eq:f(g1)f=g_2} we obtain
		\begin{equation}\label{eq:equivalence_cT_1}
			\sum_{i = 0}^{n - 1} \left( b_0^{q^{s_1(k - i)}} a_i^{q^{s_1(n - i)}} a_{j + i+k}^{q^{s_1(k - i)}}  + \sum_{r = 1}^{2} \left( c_r^{q^{s_1(k - i +r)}} a_i^{q^{s_1(n - i)}} a_{j + i + k-r}^{q^{s_1(k - i+r)}} + c_r^{q^{s_1(n - i)}} a_i^{q^{s_1(n - i)}} a_{j + i - k+r}^{q^{s_1(k-r - i)}} \right) \right)=0,
		\end{equation}
		for each $j\in\{e\ell: 0\leq \ell < k-2\text{ or  }k+2< \ell <n\}$ and $b_0,b_2\in\mathbb{F}_{q^{k}},b_1\in\mathbb{F}_{q^{2k}}$.
		
		Since $f$ is a permutation $q$-polynomial, there must be at least one nonzero coefficient $a_{i_0}$. The most important part of the rest of this proof is to show the following result:
		
		\textbf{Claim.}
		Integer $e$ satisfies $e \equiv \pm 1\pmod{n}$, and polynomial $f=a_{i_0}X^{q^{i_0}}$.

		Our strategy to prove \textbf{Claim} is to derive several necessary conditions on  $a_i$'s by choosing different value of $b_i$'s in \eqref{eq:equivalence_cT_1}.
		
		\textbf{Step.\ 1}  We take $b_0\neq 0$ and $b_1=b_2 = 0$ in \eqref{eq:equivalence_cT_1} which implies
		\begin{equation*}
			\sum_{i=0}^{k-1}b_0^{q^{s_1(k - i)}}(a_i^{q^{s_1(n - i)}} a_{j + i+k}^{q^{s_1(k - i)}} +a_{i+k}^{q^{s_1(k - i)}} a_{j + i}^{q^{s_1(n - i)}})=0.
		\end{equation*}
		As the equation above holds for any $b_0\in\fqk$,  by Lemma \ref{le:k,a_i}, we obtain
		\begin{equation}\label{equi:2_1}
			a_i^{q^{s_1(n - i)}} a_{j + i+k}^{q^{s_1(k - i)}} +a_{i+k}^{q^{s_1(k - i)}} a_{j + i}^{q^{s_1(n - i)}}=0,
		\end{equation}
		for $i\in\{0,1,\cdots, k-1\}$.
		By taking $j= 0$ in the equation above, we get
		\begin{equation}\label{equi:2_2}
			a_ia_{i+k}=0,
		\end{equation}
		for $i\in\{0,1,\cdots, k-1\}$.
		
		\textbf{Step.\ 2}  By taking $b_1\neq 0$ and $b_0=b_2 = 0$ from \eqref{eq:equivalence_cT_1}, we obtain
		\begin{equation*}
			\sum_{i=0}^{n-1}b_1^{q^{s_1(n - i)}}(a_{i-k+1}^{q^{s_1(k-i-1)}} a_{j + i}^{q^{s_1(n - i)}} +a_i^{q^{s_1(n - i)}} a_{j + i - k+1}^{q^{s_1(k-1 - i)}})=0.
		\end{equation*}
		As the  equation above holds for any $b_1\in\fqn$,  it implies
		\begin{equation}\label{equi:3_1}
			a_{i-k+1}^{q^{s_1(k-i-1)}} a_{j + i}^{q^{s_1(n - i)}} +a_i^{q^{s_1(n - i)}} a_{j -k+ i+1}^{q^{s_1(k-1 - i)}}=0.
		\end{equation}
		By taking $j= 0$ in the equation above, we have 
		\begin{equation}\label{equi:3_2}
			a_{i-k+1} a_{i}=0,
		\end{equation}
		for $i\in\{0,1,\cdots, n-1\}$.
		
		\textbf{Step.\ 3} By taking $b_1=b_0=0$ and $b_2\neq 0$ in \eqref{eq:equivalence_cT_1}, we have
		\begin{align*}
			&\sum_{i=0}^{n-1}(\eta_1 b_2)^{q^{s_1(n - i)}}(a_{i-k+2}^{q^{s_1(k-i-2)}} a_{j + i}^{q^{s_1(n - i)}} +a_i^{q^{s_1(n - i)}} a_{j + i - k+2}^{q^{s_1(k-2 - i)}})\\
			=&\sum_{i=0}^{n-1}b_2^{q^{s_1(n - i)}}(a_{i-k+2}^{q^{s_1(k-i-2)}} (\eta_1a_{j + i})^{q^{s_1(n - i)}} +(\eta_1a_i)^{q^{s_1(n - i)}} a_{j + i - k+2}^{q^{s_1(k-2 - i)}})\\
			=&\sum_{i=0}^{k-1}b_2^{q^{s_1(k - i)}}(a_{i-k+2}^{q^{s_1(k-i-2)}} (\eta_1a_{j + i})^{q^{s_1(n - i)}} +(\eta_1a_i)^{q^{s_1(n - i)}} a_{j + i - k+2}^{q^{s_1(k-2 - i)}})\\
			&+\sum_{i=0}^{k-1}b_2^{q^{s_1(k - i)}}(a_{i+2}^{q^{s_1(n-i-2)}} (\eta_1a_{j + i+k})^{q^{s_1(k - i)}} +(\eta_1a_{i+k})^{q^{s_1(k - i)}} a_{j + i+2}^{q^{s_1(n-2 - i)}})\\
			=&0.
		\end{align*}
		As the equation  above holds for any $b_2\in\fqk$,  by Lemma \ref{le:k,a_i}, we obtain
		\begin{align}\label{equi:4_1}
			\notag&	a_{i-k+2}^{q^{s_1(k-i-2)}} (\eta_1a_{j + i})^{q^{s_1(n - i)}} +(\eta_1a_i)^{q^{s_1(n - i)}} a_{j + i - k+2}^{q^{s_1(k-2 - i)}}\\&+a_{i+2}^{q^{s_1(n-i-2)}} (\eta_1a_{j + i+k})^{q^{s_1(k - i)}} +(\eta_1a_{i+k})^{q^{s_1(k - i)}} a_{j + i+2}^{q^{s_1(n-2 - i)}}=0,
		\end{align}
		for $i\in\{0,1,\cdots, k-1\}$.
		By letting $j=0$ in the equation above, we get 
		\begin{align}\label{equi:4_2}
			\notag&	a_{i-k+2}^{q^{s_1(k-i-2)}} (\eta_1a_{ i})^{q^{s_1(n - i)}} +(\eta_1a_i)^{q^{s_1(n - i)}} a_{ i - k+2}^{q^{s_1(k-2 - i)}}\\
			\notag&+a_{i+2}^{q^{s_1(n-i-2)}} (\eta_1a_{ i+k})^{q^{s_1(k - i)}} +(\eta_1a_{i+k})^{q^{s_1(k - i)}} a_{i+2}^{q^{s_1(n-2 - i)}}\\
			&=2\left(	a_{i-k+2}^{q^{s_1(k-i-2)}} (\eta_1a_{ i})^{q^{s_1(n - i)}}+(\eta_1a_{i+k})^{q^{s_1(k - i)}} a_{i+2}^{q^{s_1(n-2 - i)}}\right)=0
		\end{align}
		for $i\in\{0,1,\cdots, k-1\}$.
		
		\textbf{Step.\ 4}	
		By $a_{i_0}\neq0$ and \eqref{equi:2_2}, we get
		$a_{i_0+k}=0$. Together with \eqref{equi:3_2} we obtain $a_{i_0+k+1}=a_{i_0+k-1}=0$. From $a_{i_0+k}=0$ and \eqref{equi:4_2}, we further derive $a_{i_0+k+2}=a_{i_0+k-2}=0$. To summarize, we have proved that
		\begin{equation}\label{eq:a_1}
			a_{i_0+k}=a_{i_0+k+1}=a_{i_0+k-1}=a_{i_0+k+2}=a_{i_0+k-2}=0.
		\end{equation}
		By letting $i = i_0+k$ in \eqref{equi:2_1} and $i = i_0+k-1$, $i=i_0$ in \eqref{equi:3_1}, respectively,  thanks to \eqref{eq:a_1}, we get
		\begin{equation*}
			a_{j+i_0+k}=0,~a_{j +i_0+k-1}=0,  \text{ and } a_{j +i_0+k+1}=0
		\end{equation*}
		for  $j\in\{e\ell: 0\leq \ell < k-2\text{ or  }k+2< \ell <n\}$.
		
		By letting $i = i_0+k-2$ in \eqref{equi:4_1}, thanks to \eqref{eq:a_1}, we get
		\[
		a_{i_0}^{q^{s_1(n-i_0)}} (\eta_1a_{j + i_0+k-2})^{q^{s_1(k-i_0+2)}} 
		+(\eta_1a_{i_0-2})^{q^{s_1(2-i_0)}} a_{j + i_0+k}^{q^{s_1(k - i_0)}}=0.
		\]
		As $a_{j+i_0+k}=0$ and $a_{i_0}\neq 0$, we have
		\begin{equation*}
			a_{j+i_0+k-2}=0
		\end{equation*}
		for  $j\in\{e\ell: 0\leq \ell < k-2\text{ or  }k+2< \ell <n\}$.
		
		Furthermore, we take $i=i_0$ in \eqref{equi:4_1}. Together with \eqref{eq:a_1} and $a_{j+i_0+k}=0$, we get
		\[
		a_{j+i_0+k+2}=0.
		\]
		Therefore, we have proved the following results:
		\begin{equation}\label{eq:a_five=0}
			a_{j+i_0+k+r}=0
		\end{equation}
		for  $j\in\{e\ell: 0\leq \ell < k-2\text{ or  }k+2< \ell <n\}$ and $r\in \{0,\pm 1, \pm 2\}$.
		
		\textbf{Step.\ 5} As $a_{i_0}\neq 0$, \eqref{eq:a_five=0} implies that
		\[\ell\equiv k-re^{-1}\pmod{n}\]
		has no solution for $\ell\in\{0\leq i < k-2\}\cup\{k+2< i <n\}$ and $r\in\{0,\pm 1,\pm 2\}$. 
		It is routine to verify that $e\equiv \pm 1 \pmod{n}$ which further leads to $a_j=0$ for $j\neq i_0$. Therefore we have proved \textbf{Claim}.
		
		Now we know that $f=a_{i_0}X^{q^{i_0}}$ and $e\equiv \pm 1\pmod{n}$ which means $s_1\equiv \pm s_2\pmod{n}$.
		Taking $a_j=0$ for $j\neq i_0$ in \eqref{equi:0}, 
		we get the coefficient of $X^{q^{s_1e{(k-2)}}}$
		\begin{align}\label{eq:eta_0}
			(\eta_1^\sigma b_2)^{q^{s_1(k - i_0 +2)}} a_{i_0}^{q^{s_1(n - i_0)}} a_{i_0 -2-2e}^{q^{s_1(k - i_0+2)}} + 	(\eta_1^\sigma b_2)^{q^{s_1(n - i_0)}} a_{i_0}^{q^{s_1(n - i_0)}} a_{  i_0 -2e+2}^{q^{s_1(k-2 - i_0)}}.
		\end{align}
		
		When $e=1$, \eqref{eq:eta_0} becomes \((\eta_1^\sigma b_2)^{q^{s_1(n - i_0)}} a_{i_0}^{q^{s_1(n - i_0)}+q^{s_1(k-2 - i_0)}}\). By \eqref{eq:f(g1)f=g_2}, we have 
		\[
		\eta_2^{q^{s_1i_0}} (b_2')^{q^{s_1 i_0}} = \eta_1^\sigma b_2 a_{i_0}^{1+q^{s_1(k-2)}}.
		\]
		Hence, there exist $a\in\fqn$ and $i\in\{0,1,\cdots,n-1\}$ such that  $\eta_2^{q^{s_1i}}=a^{1+q^{s_1(k-2)}}\eta_1^\sigma$.
		
		When $e=-1$, similar computation shows that  there are $a\in\fqn$ and $i\in\{0,1,\cdots,n-1\}$ such that  $\eta_2^{q^{s_1i}}=a^{1+q^{s_1(k+2)}}\eta_1^{\sigma q^{s_1(k+2)}}$.
		
		For the proof of necessity, it is straightforward to verify the equivalence between $\cT_{n,s_1, \eta_1}$ and $\cT_{n,s_2, \eta_2}$ under (a) or (b) and we omit the proof.
	\end{proof}
	
	\section{Conclusive remarks}\label{sec:conclu}
	In this paper, we have obtained a family of maximum $\fq$-linear $(n-2)$-codes in $\sS_n(q)$ for odd prime power $q$ and $n=6,8,10$, and we can show that it is not equivalent to any previously known examples. It seems that the construction should work for any even $n\geq 6$. However, as our proof heavily relies on the choices and the computation of $(n-2)\times (n-2)$ submatrices of the associated Dickson matrix, the difficulty grows rapidly with increasing $n$. Therefore, we expect that a totally different proof of Theorem \ref{th:main}, which does not use Dickson matrices, should exist and can be generalized to the proof for any even $n\geq 6$.

		Finally, we discuss the relationship of our codes with \emph{perfect symmetric rank-metric codes} and \emph{complete symmetric rank-distance (CSRD) codes}. According to Theorem~3.5 in \cite{Mushrraf2024Perfect}, a subset \(C\) of the symmetric matrix space is a perfect code if and only if \(C\) is the whole space, or its minimum distance \(d(C)=3\), the dimension \(m\) is odd, and \(C\) is a symmetric MRD code. Our construction \(\mathcal{T}_{n,s,\eta}\) has minimum distance \(n-2\) (with \(n=6,8,10\)). Consequently, \(\mathcal{T}_{n,s,\eta}\) is not a perfect code.
		
		On the other hand, our construction \(\mathcal{T}_{n,s,\eta}\), as maximum symmetric rank-distance (MSRD) codes are necessarily  complete symmetric rank-distance (CSRD) codes. This follows directly from the definition of completeness: a code is called complete if no further codeword can be added without decreasing the minimum distance. The notion of CSRD codes has been thoroughly investigated in \cite{Alnajjarine2025Linear}, which contains a complete classification of linear CSRD codes in \(M_{3\times 3}(\mathbb{F}_{q})\).
	
	\section*{Acknowledgment}
	The authors would like to thank the two anonymous reviewers for their careful reading and helpful suggestions, which significantly improved the quality of this paper.
	This work is supported by the National Natural Science Foundation of China (No.\ 12371337) and the Natural Science Foundation of Hunan Province (No.\ 2023RC1003).
	\bibliographystyle{abbrv}
	\bibliography{references_srd} 
\end{document}